\documentclass[12pt]{article}
\usepackage{amsmath,amsfonts,amssymb,amsthm}
\usepackage{bbm}
\usepackage{mathrsfs}
\usepackage{latexsym}
\usepackage[english]{babel}
\usepackage[usenames,dvipsnames]{xcolor}
\usepackage{comment}
\usepackage{appendix}

\usepackage{graphicx}

\renewenvironment{proof}[1][\proofname]{{\bfseries #1.} }{\qed}

\def\Cov{{\rm Cov\,}}

\newcommand{\Var}{{\rm Var}}

\def\authors#1{{ \begin{center} #1 \vspace{0pt} \end{center} } \smallskip}
\def\institution#1{{\sl \begin{center} #1 \vspace{0pt} \end{center} } }
\def\inst#1{\unskip $^{#1}$}
\def\title#1{{\huge\bf  \begin{center} #1 \vspace{0pt} \end{center}  } \smallskip}

\newtheorem{theorem}{Theorem}[section]
\newtheorem{proposition}[theorem]{Proposition}
\newtheorem{lemma}[theorem]{Lemma}

\newtheorem{corollary}[theorem]{Corollary}
\newtheorem{definition}[theorem]{Definition}

\newtheorem{remark}[theorem]{Remark}

\newtheorem{assu}[theorem]{Assumption}

\newcommand{\rangleL}{{L^2(\mathbb{M}^d)}}

\newcommand{\Kspace}{\Lambda}
\newcommand{\dime}{\operatorname{dim}(\mathcal{Y}_\ell)}

\begin{document}

\date{Nov 2023}

\title{\sc Sojourn Functionals of\\ Time-Dependent $\chi^2$-Random Fields on\\ Two-Point Homogeneous Spaces}
\authors{\large Alessia Caponera\inst{*,}\footnote{E-mail address: \texttt{acaponera@luiss.it}},  Maurizia Rossi\inst{\circ,}\footnote{E-mail address: \texttt{maurizia.rossi@unimib.it}}, María Dolores Ruiz Medina\inst{\diamond,}\footnote{E-mail address: \texttt{mruiz@ugr.es}}
}
\institution{\inst{*}Dipartimento di Economia e Finanza, LUISS Guido Carli\\
\inst{\circ}Dipartimento di Matematica e Applicazioni, Università di Milano-Bicocca\\ 
\inst{\diamond}Departamento de Estadística e Investigación Operativa, Universidad de Granada 
}

\begin{abstract}

In this note we investigate geometric properties of invariant spatio-temporal random fields $X:\mathbb M^d\times \mathbb R\to \mathbb R$ defined on a compact two-point homogeneous space $\mathbb M^d$ in any dimension $d\ge 2$, and evolving over time. In particular, we focus on chi-squared distributed random fields, and study the large time behavior (as $T\to +\infty$) of the average on $[0,T]$ of the volume of the excursion set on the manifold, i.e., of $\lbrace X(\cdot, t)\ge u\rbrace$ (for any $u >0$). The Fourier components of $X$ may have short or long memory in time, i.e., integrable or non-integrable temporal covariance functions. 

Our argument follows the approach developed in \cite{MRV21} and allow to extend their results for invariant spatio-temporal Gaussian fields on the two-dimensional unit sphere to the case of chi-squared distributed fields on two-point homogeneous spaces in any dimension. We find that both the asymptotic variance and limiting distribution, as $T\to +\infty$, of the average empirical volume turn out to be non-universal, depending on the memory parameters of the field $X$.

\smallskip

\noindent\textbf{Keywords and Phrases:} Chi-square random fields; Central and nonCentral Limit Theorems; Wiener chaos; Two-point homogeneous spaces.

\smallskip

\noindent \textbf{AMS Classification:} 60G60; 60F05; 60D05; 33C45

\end{abstract}

\section{Introduction}

In this note we investigate geometric properties of invariant spatio-temporal random fields
\begin{equation*}
    X:\mathbb M^d\times \mathbb R\to \mathbb R,\qquad (x,t)\mapsto X(x,t)
  \end{equation*}  
    defined on compact two-point homogeneous spaces $\mathbb M^d$ in any dimension $d\ge 2$, and evolving over time. A key example is $\mathbb M^d=\mathbb S^d$, the $d$-dimensional sphere with the round metric. 

In the last decades the theory of spherical random fields, or more generally random fields on Riemannian manifolds, has received a growing interest, also in view of possible applications to Cosmology, Brain Imaging, Machine Learning and Climate Science, to cite a few.
As for the latter, it is indeed natural to model the Earth surface temperature as a random field on $\mathbb S^2$ evolving over time. 

\subsection{Previous work on the sphere}\label{subsec11}

In \cite{CMV23} the authors, relying on this probabilistic model, are able to detect Earth surface temperature anomalies via analysis of climate data collected from
1949 to 2020. On the other hand, in \cite{MRV21} geometric properties of invariant \emph{Gaussian} random fields $Z$ on $\mathbb S^2\times \mathbb R$ are investigated (see  also \cite{LRM23}). The Fourier components of $Z$, say $\lbrace Z_\ell \rbrace_{\ell\in \mathbb N_0}$, satisfying $\Delta_{\mathbb S^2} Z_\ell(\cdot, t) + \ell(\ell+1) Z_\ell(\cdot, t)=0$ (where $\Delta_{\mathbb S^2}$ is the Laplace-Beltrami operator on the sphere), are allowed to have short or long memory in time (independently of one another), i.e. integrable or non-integrable temporal covariance function. The parametric model is related to functions that at infinity behaves like $t^{-2}$ (or any other integrable power) in the short memory case, and $t^{-\beta_\ell}$ in the long memory case, where $\beta_\ell\in (0,1)$. In order to make the notation more compact, the short memory case is associated to $\beta_\ell=1$, representing  an integrable power law at 
infinity. For details, see \eqref{gbeta} and Assumption \ref{assu1}.

In particular, in \cite{MRV21} the authors focus on the large time (as $T\to +\infty$) fluctuations of the average empirical area 
\begin{equation}\label{MTgauss}
  \mathcal M^Z_T(u) :=  \int_{[0,T]} \int_{\mathbb S^2} (\mathbf{1}_{\lbrace Z(x,t)\ge u\rbrace } - \mathbb P(Z(x,t)\ge u))\,dx dt,
\end{equation}
where $u\in \mathbb R$ is the threshold at which we ``cut'' the field. Note that in \eqref{MTgauss} the spatial domain is fixed. 

Both first and second order fluctuations of  $\mathcal M_T^Z(u)$ depend on a subtle interplay between the parameters of the model, in particular between
\begin{equation}\label{betastarsfera}
\beta_{\ell^*} := \min_{\ell\in \mathbb N} \beta_\ell
\end{equation}
(that the authors of \cite{MRV21} assume to exist), the value of $\beta_0$ and the threshold $u$. The core of their proofs relies on the decomposition of \eqref{MTgauss} into so-called Wiener chaoses: basically, $\mathcal M_T^Z(u)$ can be written as an \emph{orthogonal} series in $L^2(\mathbb P)$ of the following form 
\begin{equation}\label{serieWienerChaos}
    \mathcal M_T^Z(u) =\sum_{q=1}^{+\infty} \mathcal M_T^Z(u)[q] =\sum_{q=1}^{+\infty} \int_{[0,T]} \left ( \frac{\phi(u) H_{q-1}(u)}{q!}\int_{\mathbb S^2} H_q(Z(x,t))\,dx\right )dt,
    \end{equation}
    where $\lbrace H_q\rbrace_{q\in \mathbb N_0}$ stands for the family of Hermite polynomials, and $\phi$ for the standard Gaussian density. 
This expansion is satisfied by any $L^2(\mathbb P)$-functional of a Gaussian field, and it is based on the fact that Hermite polynomials form a complete orthogonal basis for the space of square integrable functions on the real line w.r.t. the Gaussian measure. 

 The quantity $\beta_0\in (0,1]$ is  the memory parameter of the temporal part of $\mathcal M_T^Z(u)[1]$ in \eqref{serieWienerChaos}, i.e. of the stochastic process $\mathbb R\ni t\mapsto \phi(u) \int_{\mathbb S^2} Z(x,t)\,dx$. If $\beta_0=1$, then the asymptotic variance of $\mathcal M_T^Z(u)[1]$  is of order $T$, otherwise it is of order $T^{2-\beta_0}$. Of course,  $\mathcal M_T^Z(u)[1]$ is Gaussian for every $T>0$. 
 
 As for the second chaotic component, it is worth noticing that it vanishes if and only if $u=0$ (indeed, $H_1(u)=u$); for $u\ne 0$, roughly speaking, the dependence property of its integrand stochastic process relies on the interplay between  $\beta_{\ell^*}$ and $\beta_0$: if $2\beta_{\ell^*} < \min(1, 2\beta_0)$ then the asymptotic variance of $\mathcal M_T^Z(u)[2]$ is of order $T^{2-2\beta_{\ell^*}}$, while if both $2\beta_{\ell^*}$ and $2\beta_0$ are $>1$ it is of order $T$. In the former case, the limiting distribution of $\mathcal M_T^Z(u)[2]$ after standardization is non-Gaussian, in the latter case a Central Limit Theorem holds. The analysis of higher order chaotic components of \eqref{serieWienerChaos} is somehow analogous. 

In particular, in the ``short memory case'' (that corresponds to $\beta_0=1$ and $\beta_{\ell^*}=\min_{\ell\in \mathbb N} \beta_\ell$ s.t. $\beta_{\ell^*}> 1/2$ for $u\ne 0$ or $\beta_{\ell^*}> 1/3$ for $u= 0$)  the statistic in \eqref{MTgauss}, once divided by $\sqrt T$, shows Gaussian fluctuations. Indeed, each non-null chaotic component in \eqref{serieWienerChaos}
 has an asymptotic variance of order $T$ (so the whole variance series is of order $T$), and once divided by $\sqrt{T}$, converges to a Gaussian random variable (r.v. from now on) via an application of the Fourth Moment Theorem (hence, thanks to a Breuer-Major argument, for $\mathcal M_T^Z(u)/\sqrt T$  a Central Limit Theorem holds \cite[Theorem 6.3.1]{NP12}). 

In the ``long memory case'', in particular when  $2\beta_{\ell^*} < \min(1,\beta_0)$, for $u\ne 0$ the statistics $\mathcal M_T^Z(u)$ and $\mathcal M_T^Z(u)[2]$ are asymptotically equivalent in $L^2(\mathbb P)$, and $\mathcal M_T^Z(u)[2]$ converges in distribution towards a linear combination of  independent Rosenblatt r.v.'s, once divided by $T^{1-\beta_{\ell^*}}$ (hence \eqref{MTgauss} does not have Gaussian fluctuations). This nonCentral Limit Theorem follows from standard results for Hermite rank two functionals of Gaussian fields in presence of long memory \cite{DM79, Taq79}.

\subsection{Our main contribution}

The aim of the present note is to take a first step towards possible natural generalizations  in two directions of the results of \cite{MRV21} mentioned in Section \ref{subsec11}: the type of randomness of the field and the underlying manifold. Indeed, as briefly anticipated we want to study the large time behavior of the statistic $\mathcal M_T^X(u)$ defined as in \eqref{MTgauss} with the spherical random field $Z$ replaced by an invariant time dependent chi-squared distributed process $X$ \eqref{defX} defined on a two-point homogeneous space $\mathbb M^d$ in any dimension $d\ge 2$. Of course, we consider only the case of positive thresholds, i.e., $u>0$. (For $u=0$, it is natural to study the time behavior of the level set $\lbrace X(\cdot , t) =0\rbrace$. We leave it as a topic for future research -- see \cite{MRV24+} for the Gaussian case on the $2$-sphere.) 

To be more precise, the marginal law of $X=X_k$ is a chi-square distribution with $k\in \mathbb N$ degrees of freedom. In order to define $X_k$, it suffices to consider $k$ i.i.d. copies $Z_1,\dots, Z_k$ of an invariant time-dependent Gaussian random field $Z$ on $\mathbb M^d\times \mathbb R$ (generalizing the definition for $\mathbb S^2\times \mathbb R$), and take the sum of their squares 
\begin{equation*}
    X_k(x,t)=Z_1(x,t)^2 +\dots + Z_k(x,t)^2,\qquad (x,t)\in \mathbb M^d\times \mathbb R. 
\end{equation*}
Our main results mirror those obtained for the statistic in \eqref{MTgauss} in the spherical case. 
It is worth stressing that although our proofs closely follow the strategy developed in \cite{MRV21}, there are a few differences and we obtain some novel results that may be of independent interest. First of all, it is possible to obtain a chaotic expansion similar to \eqref{serieWienerChaos} for $\mathcal M_T^{X_k}(u)$, actually  the distribution of the underlying noise $(Z_1,\dots, Z_k)$ is Gaussian. Due to parity of the involved functions, terms of odd order $q$ vanish, and this expansion can hence be rewritten in terms of \emph{Laguerre polynomials}. We find along the way an explicit formula for chaotic coefficients of the function $\mathbf{1}_{\lbrace \chi^2(k)\ge u\rbrace}$ for $u>0$, a result that may have some independent interest. (Here and in what follows $\chi^2(k)$ denotes a chi-square r.v. with $k$ degrees of freedom.)

Let us define $\beta_{\ell^*}:=\min_{\ell\in \mathbb N_0}\beta_\ell$. As for asymptotic variance and limiting distribution of $\mathcal M_T^{X_k}(u)$, 
for $\beta_{\ell^*} > 1/2$ it is of order $T$ and a Central Limit Theorem holds for every $u>0$, while for $\beta_{\ell^*} < 1/2$ it is of order $T^{2-2\beta_{\ell^*}}$ and convergence to a linear combination of independent Rosenblatt r.v.'s holds (still for every $u>0$). 

\subsection{Brief overview of the literature}

It is worth stressing that the existing literature is rich of results on sojourn times of random fields and more generally on the geometry of their excursion sets.  In particular, in Euclidean setting the distribution of sojourn functionals for spatio-temporal Gaussian random fields with long memory has been investigated in \cite{LRM23}, while in \cite{LRMT17} the case of nonlinear functionals of Gamma-correlated random fields with long memory (and no time dependence) has been treated by means of Laguerre chaos. See also \cite{MS22} for limit theorems in case of subordinated Gaussian random fields, still with long-range dependence. The paper \cite{LMNP24} contains recent developments on this subject, indeed the authors investigate limit theorems for $p$-domain functionals of stationary Gaussian fields. Moreover, in \cite{DHLM23} the authors show that under a very general framework there is an interesting relationship between tail asymptotics of sojourn times and that of the supremum of the random field. Of course, this list is by no means complete.

To conclude, let us recall that the decomposition structure of (time varying) Gaussian random fields on compact two point homogeneous spaces was studied e.g. in \cite{MM20}, see also \cite{LMX23} for regularity properties of their sample paths.

\subsection*{Acknowledgments} 
M.R. would like to acknowledge support from the PRIN/MUR project 2022 \emph{Grafia} and the GNAMPA-INdAM project 2024 \emph{Geometria di onde aleatorie su varietà}. The research of M.D.R.M. has been partially supported by MCIN/ AEI/PID2022--142900NB-I00,  and
CEX2020-001105-M MCIN/ AEI/10.13039/501100011033.

\section{Background and notation}

Let $\mathbb M^d$ be a \emph{two-point homogeneous} compact Riemannian manifold with dimension $d\ge 2$, and denote by $\rho$ its geodesic distance. 
Here we assume that $\rho(x,y) \in [0,\pi]$, for every $x,y \in \mathbb{M}^d$, and consider also a \emph{normalized} Riemannian measure, denoted by $d\nu$, so that $\int_{\mathbb{M}^d} d\nu = 1$. Denoting by $G$ the group of isometries of $\mathbb M^d$, 
for any set of four points $x_1, y_1, x_2, y_2\in \mathbb M^d$ with $\rho(x_1, y_1)=\rho(x_2, y_2)$, there exists $\varphi \in G$ such that $\varphi(x_1)=x_2$ and $\varphi(y_1)=y_2$. As pointed out in \cite{Wan52}, $\mathbb M^d$ belongs to one of the following categories: the unit (hyper)sphere $\mathbb S^d$, $d=2,3,\dots$, the real projective spaces $\mathbb P^d(\mathbb R)$, $d=2,3,\dots$, the complex projective spaces $\mathbb P^d(\mathbb C)$, $d=4,6,\dots$, the quaternionic projective spaces $\mathbb P^d(\mathbb H)$, $d = 8, 12, \dots$, and the Cayley projective plane $\mathbb P^d(Cay)$, $d = 16$. 

\subsection{Fourier analysis for two-point homogeneous spaces}

Let $\Delta_{\mathbb{M}^d}$ be the Laplace–Beltrami operator on $\mathbb{M}^d$. It is well known that the spectrum of $\Delta_{\mathbb{M}^d}$ is purely discrete, the eigenvalues being
     \begin{equation*}
       \lambda_\ell =  \lambda_\ell (\mathbb{M}^d)  = -\ell (\ell +\alpha +\beta +1),\quad \ell \in \Kspace,
     \end{equation*}
    where
    \begin{eqnarray}\label{defLambda}
    \Kspace =  \Kspace (\mathbb{M}^d) = \begin{cases}
   \{n \in \mathbb{N}_0: n \text{ even}\} &\text{if }  \mathbb{M}^d = \mathbb{P}^d(\mathbb R)\\
   \mathbb{N}_0 &\text{if }  \mathbb{M}^d \ne \mathbb{P}^d(\mathbb R),
    \end{cases}
    \end{eqnarray}
    and the parameters $\alpha, \beta$ are reported in Table \ref{tab::alphabeta}. (In this manuscript, $\mathbb N_0=\lbrace 0,1,2,\dots \rbrace$ while $\mathbb N=\mathbb N_0\setminus \lbrace 0 \rbrace$.) 
    Now consider $L^2(\mathbb{M}^d)$, the space of (real-valued) square-integrable functions over $\mathbb{M}^d$,  endowed with the standard inner product $\langle \cdot, \cdot \rangle_{L^2(\mathbb M^d)}$.

\begin{table}[ht]
\caption{}
    \begin{center}
    \def\arraystretch{1.5}
    \begin{tabular}{c|ccccc}\hline
     $\mathbb{M}^d$ & $\mathbb S^d$ & $\mathbb{P}^d(\mathbb R)$ & $\mathbb{P}^d(\mathbb C)$ & $\mathbb{P}^d(\mathbb H)$ & $\mathbb P^d(\operatorname{Cay})$ \\ \hline
   Dimension& $d=1,2,\dots$ & $d=2,3,\dots$ & $d=4,6,\dots$ & $d=8,12,\dots$ & $d=16$ \\
   $\alpha$ & $(d-2)/2$ & $(d-2)/2$ & $(d-2)/2$ & $(d-2)/2$ & 7 \\
   $\beta$ & $(d-2)/2$ & $(d-2)/2$ &0& 1 & 3 \\
  \hline    \end{tabular}\label{tab::alphabeta}
    \end{center}
\end{table} 
   (Recall the definition of the set $\Lambda$ in \eqref{defLambda}.) For every $\ell \in \Lambda$, eigenfunctions corresponding to the same eigenvalue $\lambda_\ell$ form a finite dimensional vector space, denoted by $\mathcal{Y}_\ell = \mathcal{Y}_\ell (\mathbb{M}^d)$, with dimension 
    \begin{equation*}
\dime = \frac{(2\ell + \alpha + \beta + 1) \Gamma(\beta +1) \Gamma(\ell +\alpha + \beta +1)\Gamma(\ell +\alpha +1)}{\Gamma(\alpha +1) \Gamma(\alpha + \beta +2)\Gamma(\ell +1)\Gamma(\ell +\beta+1)}. 
\end{equation*}
  It is well known that 
\begin{equation*}
    L^2(\mathbb M^d) = \bigoplus_{\ell \in \Lambda} \mathcal{Y}_\ell,
\end{equation*}
where $\bigoplus$ stands for the orthogonal sum, in $L^2$, of vector spaces. 
Moreover, given an orthonormal basis of $\mathcal{Y}_\ell$, say $\{ Y_{\ell,m}, m=1,\dots, \dime \}$, any function $f \in L^2(\mathbb{M}^d)$ can be written as a series, converging in $L^2$, of the form 
\begin{equation*}
    f = \sum_{\ell \in \Lambda} \sum_{m =1}^{\dime} \langle f , Y_{\ell,m} \rangle_\rangleL Y_{\ell,m}.
\end{equation*}

From \cite[Theorem 3.2]{gin75} an \emph{addition formula} holds, that is, for any $\ell \in \Lambda,$
\begin{equation}\label{addition}
\frac{1}{\dime}\sum_{m=1}^{\dime } Y_{\ell, m}(x) Y_{\ell,m}(y) =   \frac{P_{\ell }^{(\alpha, \beta)}(\cos \epsilon \rho(x,y))}{P_{\ell}^{(\alpha, \beta)}(1)},
\end{equation}
where $\epsilon=1/2$ if $\mathbb{M}^d = \mathbb{P}^d(\mathbb R)$ and $\epsilon=1$ otherwise.     
   The functions $P_{\ell }^{(\alpha, \beta)}:[-1,1]\to \mathbb{R}, \ \ell \in \mathbb{N}_0,$ are the well-known \emph{Jacobi polynomials} \cite{Sze67}. Note that when $\mathbb{M}^d = \mathbb{S}^2$ is the unit two-dimensional sphere (with the round metric), they reduce to the \emph{Legendre polynomials} $P_\ell$, $\ell \in \mathbb N_0$. Recall that
   \begin{equation*}
       P_{\ell}^{(\alpha, \beta)}(1) = {\ell +\alpha \choose \ell}.
   \end{equation*}
In general, the reason for the “exceptional" behaviour of the real projective spaces is discussed for instance in \cite{Kus21} (see also \cite{KT12} and the references therein).

\subsection{Spatio-temporal invariant random fields}

Consider an isotropic and stationary centered Gaussian field $Z$ on $\mathbb M^d\times \mathbb R$, that is, a measurable map 
$Z:\Omega \times \mathbb M^d\times \mathbb R\to \mathbb R$, where $(\Omega, \mathcal F, \mathbb P)$ is some probability space, such that 
\begin{itemize}
    \item $(Z(x_1, t_1),\dots, Z(x_n, t_n))$ is Gaussian for every $n$ and every $x_1,\dots, x_n\in \mathbb M^d$, $t_1,\dots, t_n\in \mathbb R$;
    \item $\mathbb E[Z(x,t)]=0$ for every $x\in \mathbb M^d$, $t\in \mathbb R$;
    \item there exists a function $\Sigma:[-1,1]\times \mathbb R\to \mathbb R$ such that for every $x,y\in \mathbb M^d$, $t,s\in \mathbb R$,
    \begin{equation*}
        \mathbb E[Z(x,t)Z(y,s)] = \Sigma(\cos \epsilon\rho(x,y), t-s),
    \end{equation*}
    where $\epsilon=1/2$ if $\mathbb{M}^d = \mathbb{P}^d(\mathbb R)$ and $\epsilon=1$ otherwise. 
\end{itemize}
Throughout this note we also assume that $Z$ is mean-square continuous, that is, $\Sigma$ is a continuous function.

It is well-known \cite{MM20} that $\Sigma$ admits the following spectral decomposition
\begin{equation}\label{seriesGamma}
    \Sigma( \theta, \tau) = \sum_{\ell\in \Lambda} \kappa_\ell C_\ell(\tau) P^{(\alpha, \beta)}_{\ell}(\cos \epsilon \theta), \quad \tau\in \mathbb R, \theta\in [0,\pi], 
\end{equation}
where 
\begin{equation*}
    \kappa_\ell = \frac{\dime}{P_{\ell}^{(\alpha, \beta)}(1)},
\end{equation*}
the series on the right hand side (r.h.s. from now on) of \eqref{seriesGamma} converges uniformly, and $C_\ell:\mathbb R\to \mathbb R$, $\ell\in \Lambda$, are semi-positive definite functions. Note that the uniform convergence of the series in \eqref{seriesGamma} is equivalent to the following condition: $
\sum_{\ell\in \mathbb N_0} \dime C_\ell(0) <+\infty$. 

Throughout this work we assume that the field $Z$ has unit variance, i.e., 
\begin{equation}\label{var1}
   \sum_{\ell\in \Lambda}  \dime C_\ell(0)   = 1. 
\end{equation}
\begin{definition}\label{defX}
Let $Z_1, Z_2,\dots, Z_k$ be $k\in \mathbb N$ i.i.d. copies of an isotropic and stationary centered Gaussian field $Z$ on $\mathbb M^d\times \mathbb R$. For $k\in \mathbb N$, we
define the spatio-temporal $\chi^2(k)$-random field $X_k$ as 
\begin{equation}
X_k(x,t) := Z_1^2(x,t) +\dots + Z_k^2(x,t),\quad x\in \mathbb M^d, t\in \mathbb R.
\end{equation} 
\end{definition}
Note that the marginal distribution of $X_k$ is a chi-square probability law with $k$ degrees of freedom. Hence $\mathbb E[X_k(x,t)]=k$, and moreover\footnote{Remark that $\Cov(X_k(x,t),X_k(y,s)) = k\,\Cov(Z(x,t)^2,Z(y,s)^2) = k\, \mathbb E[H_2(Z(x,t))H_2(Z(y,s))]= 2k\,\Cov(Z(x,t), Z(y,s))^2$, where $H_2(r)=r^2-1$, $r\in \mathbb R$, is the $2$-nd Hermite polynomial.},  
\begin{align*}
    \Cov(X_k(x,t),X_k(y,s)) = 2k\, \Sigma(\cos \epsilon \rho(x,y), t-s)^2,
\end{align*}
where $\epsilon=1/2$ if $\mathbb{M}^d = \mathbb{P}^d(\mathbb R)$ and $\epsilon=1$ otherwise. 

\subsubsection{Short and long memory}

Let $Z$ be a (mean square continuous) isotropic and stationary centered Gaussian field $Z$ on $\mathbb M^d\times \mathbb R$.
From \eqref{seriesGamma}, it is straightforward to prove that $Z$ admits the following spectral decomposition 
\begin{align}\label{seriesZ}
    Z(x,t) = \sum_{\ell\in \Lambda} \sum_{m=1}^{\dime} a_{\ell,m}(t) Y_{\ell,m}(x),
\end{align}
where $a_{\ell,m}, \ell\in \Lambda, m=1,\dots ,\dime$ are independent stationary (centered) Gaussian stochastic processes, indexed by $\mathbb R$, such that 
\begin{equation*}
    \mathbb E[a_{\ell,m}(t)a_{\ell',m'}(s)] = \delta_\ell^{\ell'}\delta_{m}^{m'} C_\ell(t-s).
\end{equation*}
Note that the convergence of the series in \eqref{seriesZ} is in $L^2(\mathbb M^d\times [-T,T])$ for any $T>0$. Of course, from now on we can restrict ourselves to $\tilde \Lambda:=\lbrace \ell \in \Lambda : C_\ell(0)\ne 0\rbrace$.

Similarly to the spatio-temporal model on $\mathbb S^2$ defined in \cite{MRV21}, we allow the stochastic processes $a_{\ell,m}$, $\ell\in \tilde \Lambda$, $m=1,\dots, \dime$ to have short or long range dependence. Let us define the family of functions $g_\beta$, $\beta\in (0,1]$, as follows: for $\tau\in \mathbb R$, 
\begin{equation}\label{gbeta}
    g_\beta(\tau) := \begin{cases}
        (1+|\tau|)^{-\beta},\qquad &\beta\in (0,1),\\
        (1+|\tau|)^{-2},\qquad &\beta=1.
    \end{cases}
\end{equation}
\begin{assu}\label{assu1}
    The covariance functions $C_\ell$, $\ell\in \tilde \Lambda$, can be written as
    \begin{equation*}
        C_\ell(\tau) = G_\ell(\tau)\cdot g_{\beta_\ell}, \qquad \tau\in \mathbb R,
    \end{equation*}
    for some parameters $\beta_\ell\in (0,1]$, $\ell\in \tilde \Lambda$, where 
    \begin{equation*}
\sup_{\ell\in \tilde \Lambda} \left |\frac{G_\ell(\tau)}{G_\ell(0)}-1 \right | = o(1),\quad \tau\to +\infty. 
    \end{equation*}
\end{assu}
Note that $G_\ell(0)=C_\ell(0)$. From now on we assume that also Assumption \ref{assu1} holds. If $\beta_\ell=1$ (resp. $\beta_\ell <1$), the processes $a_{\ell,m}$, $m=1,\dots,\dime$, have short (resp. long) memory, indeed  their common covariance function $C_\ell$ is integrable on $\mathbb R$ (resp. $\int_{\mathbb R } |C_\ell| = +\infty$), see \eqref{gbeta}.
\begin{assu}\label{assu2}
There exists 
\begin{equation}\label{betastar}
   \beta_{\ell^*} := \min_{\ell\in \tilde \Lambda} \beta_\ell. 
\end{equation}
\end{assu}
From now on we assume that also Assumption \ref{assu2} holds. In particular, $\beta_{\ell^*}\in (0,1]$. (Note that the definition of $\beta_{\ell^*}$ in \eqref{betastar} is slightly different than the one given in \cite{MRV21} for the spherical case, actually in that paper the minimum is taken over $\ell \in\tilde \Lambda(\mathbb S^2) \setminus \lbrace 0\rbrace$, while here we consider $\ell\in \tilde \Lambda(\mathbb M^d)$. The reason relies on the fact that in our case the first chaotic component $\mathcal M^{X_k}_T(u)[1]$ vanishes for every $u>0$, in contrast with the spherical Gaussian case where the integrand process in $\mathcal M^Z_T(u)[1]$ has memory parameter $\beta_0$.) 

\section{The sojourn functional}

For $u>0$, let us consider the following geometric functional of the spatio-temporal $\chi^2(k)$-random field $X_k$ on $\mathbb M^d\times \mathbb R$ (as in Definition \ref{defX}): for any $T>0$, 
\begin{equation}\label{defM}
    \mathcal M_T(u) =\mathcal M_T^{X_k}(u):= \int_{[0,T]} \int_{\mathbb M^d} (\mathbf{1}_{\lbrace X_k(x,t)\ge u\rbrace}-\mathbb P(\chi^2(k) \ge u))\,d\nu(x) dt,
\end{equation}
where $\mathbb P(\chi^2(k) \ge u)$ is the tail distribution of a chi-square r.v. with $k$ degrees of freedom. Note that $\mathbb E[\mathcal M_T(u)]=0$. 

The random variable $\mathcal M_T(u)$ is a square integrable functional of the underlying Gaussian field $(Z_1, Z_2, \dots, Z_k)$, hence it admits a Wiener-It\^o chaos expansion \cite[Section 2.2]{NP12} of the form
\begin{equation}\label{seriesWiener}
    \mathcal M_T(u) = \sum_{q=1}^{+\infty}\mathcal M_T(u)[q],
\end{equation}
where the series converges in $L^2(\mathbb P)$ and $\mathcal M_T(u)[q], \mathcal M_T(u)[q']$ are orthogonal in $L^2(\mathbb P)$ whenever $q\ne q'$. 
Moreover, for $q\ge 1$,
\begin{align*}
    \mathcal M_T(u)[q] =  \sum_{\substack{n_1,\dots, n_k\in \mathbb N_0\\ n_1 +\dots + n_k =q}}\frac{\alpha_{n_1, \dots, n_k}(u)}{n_1!\cdots n_k!} \int_{[0, T]}\int_{\mathbb M^d} \prod_{j=1}^k H_{n_j}(Z_j(x,t))\,d\nu(x) dt,
    \end{align*}
    where for $N_1,\dots,N_k$ i.i.d. standard Gaussians, 
    \begin{align}\label{alpha}
    \alpha_{n_1, \dots, n_k}(u) = \mathbb E \left [\mathbf{1}_{\lbrace \sum_{j=1}^k N_j^2\ge u\rbrace } \prod_{j=1}^k H_{n_j}(N_j)\right ]. 
    \end{align}
(Note that the coefficients in \eqref{alpha} are \emph{symmetric}.)    Here $H_q$, $q\in \mathbb N_0$, denotes the $q$-th Hermite polynomial \cite[Section 1.4]{NP12}, that can be defined as follows: $H_0\equiv 1$, and for $q\ge 1$, $H_q(r) = (-1)^q \phi^{-1}(r) \frac{d^q}{dr^q}\phi(r)$, $r\in \mathbb R$, $\phi$ denoting the standard Gaussian density. Recall that $H_q,q\in \mathbb N_0$, form an orthogonal basis for the space of square integrable real-valued functions with respect to the Gaussian measure $\phi(r)dr$. 
\begin{lemma}\label{lemodd}
Let $k\in \mathbb N$ and $n_1,\dots,n_k\in \mathbb N_0$. If there exists $j\in \lbrace 1,2,\dots, k\rbrace$ such that $n_j$ is an odd number, then    \begin{equation*}
       \alpha_{n_1, \dots, n_k}(u)=0.
       \end{equation*} 
    \end{lemma}
    \begin{proof}
For $k=1$ it is straightforward due to parity of Hermite polynomials.  
For $k\ge 2$ we can assume w.l.o.g. that $n_1$ is \emph{odd}, then since 
\begin{align*}
\alpha_{n_1, \dots, n_k}(u) &= \mathbb E \left [ \mathbb E \left [1_{\sum_{j=1}^k N_j^2\ge u} \prod_{j=1}^k H_{n_j}(N_j) \Big | N_{2},\dots, N_{k} \right ]\right ] \\
&= \mathbb E \left [ \prod_{j=2}^k H_{n_j}(N_j){\mathbb E \left [1_{ N_1^2 + \sum_{j=2}^k x_j^2\ge u}  H_{n_1}(N_1) \right ]}_{ \Big |x_2=N_{2},\dots, x_k=N_{k}} \right ]
\end{align*}
and $\mathbb E[1_{N_1^2\ge r} H_{n_1}(N_1)] =0$ for every $r\in \mathbb R$,
we deduce that $\alpha_{n_1, \dots, n_k}(u)=0$. 

\end{proof}

Thus \eqref{seriesWiener} can be rewritten as follows. 
\begin{lemma}\label{lemcaos}
The chaotic expansion of $\mathcal M_T(u)$ in \eqref{defM} is 
\begin{align*}
    \mathcal M_T(u) &= \sum_{q=1}^{+\infty}\mathcal M_T(u)[2q]\\
    &=\sum_{q=1}^{+\infty} \sum_{n_1 +\dots + n_k =q}\frac{\alpha_{2n_1, \dots, 2n_k}(u)}{(2n_1)!\cdots (2n_k)!} \int_{[0,T]} \int_{\mathbb M^d} \prod_{j=1}^k H_{2n_j}(Z_j(x,t))\,d\nu(x) dt,
    \end{align*}
    where the series is orthogonal and converges in $L^2(\mathbb P)$.
    \end{lemma}
    (The proof of Lemma \ref{lemcaos} closely follows the proof of Lemma 4.1 in \cite{MRV21}, hence we omit the details.) 
The chaotic expansion in Lemma \ref{lemcaos} can be rewritten in terms of Laguerre polynomials thanks to  \cite[ (22.5.40), (22.5.18)]{AS64}. 

     Note that, by standard properties of Hermite polynomials \cite[Section 1.4]{NP12}, 
\begin{equation}\label{seriesfin}
  \Var(1_{\chi^2(k)\ge u}) =  \mathbb P(\chi^2(k) \ge u)(1-\mathbb P(\chi^2(k) \ge u))=\sum_{q=1}^{+\infty} \sum_{n_1 +\dots + n_k =q}\frac{(\alpha_{2n_1, \dots, 2n_k}(u))^2}{(2n_1)!\cdots (2n_k)!},
\end{equation}
thus the series on the r.h.s. of \eqref{seriesfin} converges. 
\begin{lemma}\label{lemalpha} 
Let $u>0$, $q\in \mathbb N$ and $n_1,\dots, n_k\in \mathbb N_0$ be such that $n_1 +\dots + n_k =q$. 
Then for $k=1$
\begin{equation*}
    \alpha_{2n_1}(u) =  2 \phi(\sqrt u) H_{2n_1-1}(\sqrt u),
\end{equation*}
where $\phi$ is the standard Gaussian density,
while for $k\ge 2$
\begin{eqnarray}
    \alpha_{2n_1, \dots, 2n_k}(u)
    = &&\sum_{0\le r_1\le n_1, \dots,0\le r_k\le n_k} \frac{(-1)^{r_1+\dots + r_k}}{2^{n_1}\cdots 2^{n_k}}     \frac{(2n_1)!\cdots  (2n_k)!}{r_1!(n_1-r_1)! \cdots r_k!(n_k-r_k)!}\cr
    &&\times  \frac{\gamma(\frac{k}{2} + n_1-r_1+\dots +n_k-r_k, \frac{u}{2})}{\Gamma(\frac{k}{2}+n_1-r_1+\dots +n_k-r_k)},
\end{eqnarray}
where $\gamma$ is the lower incomplete Gamma function \cite[Section 6.5]{AS64}, and $\Gamma$ is the (standard) Gamma function. 
\end{lemma}
Lemma \ref{lemalpha} immediately entails that $\alpha_{2}(u) = 2\phi(\sqrt u) \sqrt u> 0$ for every $u>0$. Actually, we are able to show that $\alpha_{2n_1, 2n_2, \dots, 2n_k}(u)> 0$ for $n_1=1$, $n_2=\dots =n_k=0$, for all $k\ge 1$ and $u>0$, i.e., 
     \begin{equation}\label{alpha2}
         \alpha_{2,0,\dots,0}(u) > 0,\qquad \forall u>0.
     \end{equation}
Lemma \ref{lemalpha} is new, its proof is technical hence we collect it in Appendix \ref{secA1}, together with the proof of \eqref{alpha2}.

    \section{Main results}

We are interested in the large time behavior of the average empirical measure \eqref{defM} at any positive threshold. We will observe a phase transition for both asymptotic variance and limiting distribution depending on the memory parameters of the field, in particular under the ``long-memory regime", i.e., for $2\beta_{\ell^*} < 1$ (where $\beta_{\ell^*}$ is defined as in \eqref{betastar}), we will observe a non-Gaussian behavior for the second order fluctuations of $\mathcal M_T(u)$ as $T\to +\infty$, in contrast with the central limit theorem that we will prove in the case of ``short-memory", i.e., for $2\beta_{\ell^*} > 1$. 

In order to state our main results, we first need some more notation. 
\begin{definition}\label{defRosenblatt}
The r.v. $X_\beta$ has the standard Rosenblatt distribution (see
e.g. \cite{Taq75} and also ~\cite{DM79,Taq79}) with parameter $\beta \in (0,%
\frac{1}{2})$ if it can be written as
\begin{equation}  \label{Xbeta}
X_{\beta}=a(\beta) \int_{(\mathbb{R}^{2})^{^{\prime }}}\frac{e^{i(\lambda
_{1}+\lambda _{2})}-1}{i(\lambda _{1}+\lambda _{2})}\frac{W(d\lambda
_{1})W(d\lambda _{2})}{|\lambda _{1}\lambda _{2}|^{(1-\beta )/2}}\text{ ,}
\end{equation}%
where $W$ is the white noise Gaussian measure on $\mathbb{R}$, the
stochastic integral is defined in the Ito's sense (excluding the diagonals:
as usual, $(\mathbb{R}^2)^{\prime }$ stands for the set $\lbrace (\lambda_1,
\lambda_2)\in \mathbb{R}^2:\lambda_1 \ne \lambda_2\rbrace$), and
\begin{equation}  \label{abeta}
a(\beta) := \frac{\sigma(\beta)}{2\,\Gamma(\beta)\,\sin\left({(1-\beta)\pi}/{%
2}\right)}\,,
\end{equation}
with
\begin{equation*}
\sigma(\beta) := \sqrt{\frac12(1-2\beta)(1-\beta)}\,.
\end{equation*}
Following \cite{MRV21}, we say the random vector $V$ satisfies a composite
Rosenblatt distribution of degree $N\in \mathbb{N}$ with parameter $\beta$ and coefficients $%
c_{1},...,c_{N}\in \mathbb{R}\setminus \lbrace 0\rbrace,$ if%
\begin{equation}  \label{V}
V=V_{N}(c_{1},...,c_{N};\beta )\mathop{=}^d\sum_{k=1}^{N}c_{k}X_{k;\beta}%
\text{ ,}
\end{equation}%
where $\left\{ X_{k;\beta}\right\} _{k=1,...,N}$ is a collection of i.i.d.
standard Rosenblatt r.v.'s \eqref{Xbeta} of parameter $\beta $.
\end{definition}

Note that indeed $\mathbb{E}[X_\beta]=0$ and $\mathop{\rm Var}(X_\beta)=1$.
The Rosenblatt distribution was first introduced in \cite{Taq75} and has
already appeared in the context of spherical isotropic Gaussian random
fields as the exact distribution of the correlogramm, see \cite{LTT18}.
\begin{remark}\rm 
\textrm{The characteristic function $\Xi_V$ of $V=V_{N}(c_{1},...,c_{N};%
\beta )$ in \eqref{V} is 
\begin{equation*}
\Xi_V(\theta) = \prod_{k=1}^N \xi_\beta(c_k \theta),\qquad
\xi_\beta(\theta)= \exp\left ( \frac12 \sum_{j=2}^{+\infty} \left ( 2i
\theta \sigma(\beta) \right )^j \frac{a_j}{j} \right ),
\end{equation*}
where $\xi_\beta$ is the characteristic function of $X_\beta$ in %
\eqref{Xbeta}, the series is only convergent near the origin and
\begin{equation*}
a_j := \int_{[0,1]^j} |x_1 - x_2|^{-\beta} |x_2 - x_3|^{-\beta}\cdots
|x_{j-1} - x_j|^{-\beta} |x_j - x_1|^{-\beta}dx_1 dx_2 \cdots dx_j .
\end{equation*}
Note that when $\beta\to 0^+$ (resp. $\beta \to \frac12^-$) then $\xi_\beta$ approaches the characteristic
function of $(Z^2 -1)/\sqrt{2}$ (resp. $Z$), where $Z\sim \mathcal{N}(0,1)$ is a
standard Gaussian random variable.  Many more characterizations of the
Rosenblatt distribution have been given in the literature: for instance its
infinite divisibility properties and Levy-Khinchin representation are
discussed in \cite{LRT17} and the references therein. }
\end{remark}
\begin{theorem}\label{th1}
For $2\beta_{\ell^*}<1$, as $T\to +\infty$, 
\begin{equation}
   \lim_{T\to +\infty} \frac{\textnormal{Var}(\mathcal M_T(u))}{T^{2-2\beta_{\ell^*}}} =
         k\frac{(\alpha_{2,0,\dots,0}(u))^2}{2 (2-2\beta_{\ell^*})(1-2\beta_{\ell^*})}\sum_{\ell\in \mathcal I^*} \dime (C_\ell(0))^2 =:v_{\beta_{\ell^*}}(u).
         \label{varLRD}
\end{equation}
Assume moreover that the set 
\begin{equation}\label{setstar}
    \mathcal I^*:=\lbrace \ell\in \tilde \Lambda : \beta_\ell=\beta_{\ell^*}\rbrace 
\end{equation}
is finite.  Then,
\begin{equation}\label{rosU}
    \frac{\mathcal M_T(u)}{\sqrt{\textnormal{Var}(\mathcal M_T(u))}} \mathop{\longrightarrow}^d \mathcal R(u),
\end{equation}
where 
$$\mathcal R(u)\mathop{=}^d  \frac{\alpha_{2, 0,\dots, 0}(u)}{2!}\sum_{j=1}^k \sum_{\ell_j\in \mathcal I^*} \frac{C_{\ell_j}(0)}{a(\beta_{\ell^*})\sqrt{v_{\beta_{\ell^*}}(u)}} \sum_{m_j=1}^{\textnormal{dim}(\mathcal Y_{\ell_j})}X_{m_j;\beta_{\ell^*}}
$$
is a composite Rosenblatt r.v. with parameter $\beta_{\ell^*}$ and degree $\sum_{j=1}^k\textnormal{dim}(\mathcal Y_{\ell_j})$ (see Definition \ref{defRosenblatt}).
\end{theorem}
Note that the r.h.s. of \eqref{varLRD} is (strictly) positive and finite, indeed $\alpha_{2,0,\dots,0}(u) >0$ from \eqref{alpha2}, $\beta_{\ell^*}\in (0,\frac12)$ and $0<\sum_{\ell\in \mathcal I^*} \dime C_\ell(0)^2 < +\infty$ (actually, from \eqref{var1} we deduce that $C_\ell(0) \le 1$ $\forall \ell\in \Lambda$). 

   \begin{theorem}\label{th2}
For $2\beta_{\ell^*}>1$, as $T\to +\infty$, 
\begin{equation}\label{varSRD}
    \lim_{T\to +\infty}\frac{\textnormal{Var}(\mathcal M_T(u))}{T} = 
         \sum_{q=1}^{+\infty} s^2_{2q},
\end{equation}
where 
\begin{eqnarray}\label{s2}
s^2_2&=& k  \frac{\left (\alpha_{2,0, \dots, 0}(u)\right )^2}{2!} \sum_{\ell\in \tilde \Lambda} \dime \int_{\mathbb R}  (C_{\ell}(\tau))^2\,d\tau,
\end{eqnarray}
and for $q\ge 2$
\begin{eqnarray}\label{s2q}
s^2_{2q} &:=& \sum_{n_1 +\dots + n_k =q}  \frac{\left (\alpha_{2n_1, \dots, 2n_k}(u)\right )^2}{(2n_1)!\cdots (2n_k)!} \sum_{\ell_1,\dots,\ell_{2q}\in \tilde \Lambda} \kappa_{\ell_1}\cdots \kappa_{\ell_{2q}} \int_{\mathbb R}  C_{\ell_1}(\tau)\cdots C_{\ell_{2q}}(\tau)\,d\tau\cr 
        &&\times \int_{(\mathbb M^d)^2} P^{(\alpha, \beta)}_{\ell_1}(\cos \epsilon \rho(x,y))\cdots P^{(\alpha, \beta)}_{\ell_{2q}}(\cos \epsilon \rho(x,y))\,d\nu(x)d\nu(y).
    \end{eqnarray}
Moreover 
\begin{equation}
    \frac{\mathcal M_T(u)}{\sqrt{\textnormal{Var}(\mathcal M_T(u))}} \mathop{\longrightarrow}^d Z,
\end{equation}
where $Z\sim \mathcal N(0,1)$ is a standard Gaussian r.v.
\end{theorem}
As for the asymptotic variance \eqref{varSRD}, note that $0< \sum_{q=1}^{+\infty} s^2_{2q} < +\infty$. Actually, $s_2^2 >0$, while the convergence of the series follows from the inequality 
\begin{equation*}
    \int_{(\mathbb M^d)^2} |P^{(\alpha, \beta)}_{\ell_1}(\cos \epsilon \rho(x,y))\cdots P^{(\alpha, \beta)}_{\ell_{2q}}(\cos \epsilon \rho(x,y))|\,d\nu(x)d\nu(y) \le P^{(\alpha, \beta)}_{\ell_1}(1)\cdots P^{(\alpha, \beta)}_{\ell_{2q}}(1),
\end{equation*}
Lemmas 4.13--4.17 in \cite{MRV21} and \eqref{var1}.

\subsection{Outline of the proofs}

The nature of our results in Theorems \ref{th1} and \ref{th2} relies on the non-universal asymptotic behavior of chaotic components in \eqref{seriesWiener}. Their proofs are fully inspired by \cite{MRV21}. 
\begin{proposition}\label{prop1}
   Let $q\ge 1$. For $2q\beta_\ell^* <1$ 
    \begin{eqnarray*}
        \lim_{T\to +\infty}\frac{\Var(\mathcal M_T(u)[2q])}{T^{2-2q\beta_{\ell^*}}} &=&\sum_{n_1 +\dots + n_k =q}  \frac{\left (\alpha_{2n_1, \dots, 2n_k}(u)\right )^2}{(2n_1)!\cdots (2n_k)!} \sum_{\ell_1,\dots,\ell_{2q}\in \mathcal I^*}^{+\infty} \kappa_{\ell_1}\cdots \kappa_{\ell_{2q}}   \frac{C_{\ell_1}(0)\cdots C_{\ell_{2q}}(0)}{(1-2q\beta_{\ell^*})(2-2q\beta_{\ell^*})}\cr 
        &&\times \int_{(\mathbb M^d)^2} P^{(\alpha, \beta)}_{\ell_1}(\cos \epsilon \rho(x,y))\cdots P^{(\alpha, \beta)}_{\ell_{2q}}(\cos \epsilon \rho(x,y))\,d\nu(x)d\nu(y),
    \end{eqnarray*}
    while for $2q\beta_\ell^* =1$
    \begin{eqnarray*}
        \lim_{T\to +\infty}\frac{\Var(\mathcal M_T(u)[2q])}{T\log T} &=& \sum_{n_1 +\dots + n_k =q}  \frac{\left (\alpha_{2n_1, \dots, 2n_k}(u)\right )^2}{(2n_1)!\cdots (2n_k)!} \sum_{\ell_1,\dots,\ell_{2q}\in \mathcal I^*}^{+\infty} \kappa_{\ell_1}\cdots \kappa_{\ell_{2q}}   2C_{\ell_1}(0)\cdots C_{\ell_{2q}}(0)\cr 
        &&\times \int_{(\mathbb M^d)^2} P^{(\alpha, \beta)}_{\ell_1}(\cos \epsilon \rho(x,y))\cdots P^{(\alpha, \beta)}_{\ell_{2q}}(\cos \epsilon \rho(x,y))\,d\nu(x)d\nu(y),
    \end{eqnarray*}
    and finally for $2q\beta_\ell^* >1$
    \begin{equation*}
        \lim_{T\to +\infty}\frac{\Var(\mathcal M_T(u)[2q])}{T} = s^2_{2q},
    \end{equation*}
    where $s^2_{2q}$ is as in \eqref{s2} for $q=1$ and \eqref{s2q} for $q\ge 2$. 
\end{proposition}
The proof of Proposition \ref{prop1} is postponed to Section \ref{subsecAV}.
Proposition \ref{prop1} implies the following.
\begin{corollary}
    For $2\beta_{\ell^*}<1$, as $T\to +\infty$,
\begin{equation}\label{2chaoswin}
    \frac{\mathcal M_T(u)}{\sqrt{\Var(\mathcal M_T(u))}} = \frac{\mathcal M_T(u)[2]}{\sqrt{\Var(\mathcal M_T(u)[2])}} + o(1),
\end{equation}
where $o(1)$ denotes a sequence of r.v.'s converging to zero in $L^2(\mathbb P)$. 
\end{corollary}
\begin{proposition}\label{law2caos}
For $2\beta_{\ell^*}<1$, as $T\to +\infty$,
\begin{equation*}
    \frac{\mathcal M_T(u)[2]}{\sqrt{\Var(\mathcal M_T(u)[2])}} \mathop{\longrightarrow}^d \mathcal R(u),
\end{equation*}
where $\mathcal R(u)$ is as in \eqref{rosU}. 
\end{proposition}
 The proof of Proposition \ref{law2caos}  is based on non-central limit theorems for Hermite transforms of long memory stochastic processes \cite{DM79}. Proposition \ref{law2caos} together with \eqref{2chaoswin} allow to prove Theorem \ref{th1}.

From Proposition \ref{prop1}, for $2\beta_{\ell^*}>1$, the order of magnitude of $\Var(\mathcal M_T(u)[2q])$ is $T$ for every $q\ge 1$ (since $2q\beta_{\ell^*} >1$ for every $q\ge 1$). 
\begin{proposition}\label{prop2}
    For any $q\ge 1$ s.t. $2q\beta_{\ell^*}>1$, 
    \begin{equation}
        \frac{\mathcal M_T(u)[2q]}{\sqrt{\Var(\mathcal M_T(u)[2q])}} \mathop{\longrightarrow}^d Z,
    \end{equation}
    where $Z\sim \mathcal N(0,1)$.
\end{proposition}
The proof of Theorem \ref{th2} is then an application of the central limit theorem à la Breuer-Major via chaotic expansion \cite[Theorem 6.3.1]{NP12}, whose key intermediate step is indeed Proposition \ref{prop2}.

    \section{Proofs of the main results}

    \subsection{Asymptotic variance}\label{subsecAV}
    
    Let us study the variance of each chaotic component in \eqref{seriesWiener}. By independence and equidistribution of the fields $Z_1,\dots, Z_k$, and standard properties of Hermite polynomials \cite[Section 1.4]{NP12}, we can write, for any $q\ge 1$,
    \begin{eqnarray}
        \Var(\mathcal M_T(u)[2q]) &=& \sum_{\substack{n_1 +\dots + n_k =q\\ n_1'+\dots + n_k'=q}}\frac{\alpha_{2n_1, \dots, 2n_k}(u)}{(2n_1)!\cdots (2n_k)!}\frac{\alpha_{2n'_1, \dots, 2n'_k}(u)}{(2n'_1)!\cdots (2n'_k)!}\times \cr
        &\times & \int_{0}^T \int_{0}^T \int_{\mathbb M^d}\int_{\mathbb M^d} 
       \prod_{j=1}^k \mathbb E\left [ H_{2n_j}(Z_j(x,t)) H_{2n'_j}(Z_j(y,s))\right ]\,d\nu(x) d\nu(y) dt ds
        \cr
        &=& \sum_{n_1 +\dots + n_k =q}  \frac{\left (\alpha_{2n_1, \dots, 2n_k}(u)\right )^2}{(2n_1)!\cdots (2n_k)!} \int_{[0,T]^2}  \int_{(\mathbb M^d)^2} (\Sigma(\cos \epsilon\rho(x,y), t-s))^{2q}\,d\nu(x) d\nu(y) dt ds\cr 
        &=&  \sum_{n_1 +\dots + n_k =q}  \frac{\left (\alpha_{2n_1, \dots, 2n_k}(u)\right )^2}{(2n_1)!\cdots (2n_k)!} \sum_{\ell_1,\dots,\ell_{2q}\in \tilde \Lambda}^{+\infty} \kappa_{\ell_1}\cdots \kappa_{\ell_{2q}} \int_{[0,T]^2}  C_{\ell_1}(t-s)\cdots C_{\ell_{2q}}(t-s)\,dtds\cr 
        &\times & \int_{(\mathbb M^d)^2} P^{(\alpha, \beta)}_{\ell_1}(\cos \epsilon \rho(x,y))\cdots P^{(\alpha, \beta)}_{\ell_{2q}}(\cos \epsilon \rho(x,y))\,d\nu(x)d\nu(y).\label{var2qcaos}
        \end{eqnarray}
        In particular, for $q=1$, thanks to the addition formula \eqref{addition} and \emph{symmetry} of chaotic coefficients 
        \begin{eqnarray}
            \Var(\mathcal M_T(u)[2]) &=&  k\frac{(\alpha_{2,0,\dots,0}(u))^2}{2!}\sum_{\ell\in \tilde \Lambda} \kappa_\ell^2 \frac{(P_{\ell}^{(\alpha,\beta)}(1))^2}{\dime}\int_{[0,T]^2} (C_\ell(t-s))^2\,dt ds\cr 
            &=&  k\frac{(\alpha_{2,0,\dots,0}(u))^2}{2!}\sum_{\ell\in \tilde \Lambda} \dime\int_{[0,T]^2} (C_\ell(t-s))^2\,dt ds.\label{var2caos}
            \end{eqnarray}
        Let us set for notational convenience, for $\ell\in \tilde \Lambda$,
        \begin{equation}\label{Kell}
            K_{\ell}(T):= \int_{[0,T]^2} (C_\ell(t-s))^2\,dt ds,
            \end{equation}
        and for $q\ge 2$, $\ell_1,\dots, \ell_{2q}\in \tilde \Lambda$,
        \begin{equation}\label{Kell12}
            K_{\ell_1,\dots, \ell_{2q}}(T) := \int_{[-T,T]^2}  C_{\ell_1}(t-s)\cdots C_{\ell_{2q}}(t-s)\,dtds.        \end{equation}
        \begin{proof}[Proof of Proposition \ref{prop1}]
        From Lemma 4.4 in \cite{MRV21}, for $\ell\in \tilde \Lambda$,
        if $2\beta_\ell < 1$ 
        \begin{equation*}
           \lim_{T\to +\infty}\frac{K_{\ell}(T)}{T^{2-2\beta_{\ell}}}  = \frac{(C_{\ell}(0))^2}{(1-2\beta_{\ell}))(2-2\beta_{\ell})},          \end{equation*}
while for $2\beta_\ell = 1$ 
        \begin{equation*}
           \lim_{T\to +\infty}\frac{K_{\ell}(T)}{T\log T}  = 2 (C_{\ell}(0))^2.        \end{equation*}
Finally for $2\beta_{\ell} > 1$             \begin{equation*}
           \lim_{T\to +\infty}\frac{K_{\ell}(T)}{T}  = \int_{\mathbb R} (C_{\ell}(\tau))^2\,d\tau. 
           \end{equation*}         
        From Lemma 4.11 in \cite{MRV21}, for every $q\ge2$, for $\beta_{\ell_1}+\dots +\beta_{\ell_{2q}} < 1$ 
        \begin{equation*}
           \lim_{T\to +\infty}\frac{K_{\ell_1,\dots, \ell_{2q}}(T)}{T^{2-(\beta_{\ell_1}+\dots +\beta_{\ell_{2q}})}}  = \frac{C_{\ell_1}(0)\cdots C_{\ell_{2q}}(0)}{(1-(\beta_{\ell_1}+\dots +\beta_{\ell_{2q}}))(2-(\beta_{\ell_1}+\dots +\beta_{\ell_{2q}}))},          \end{equation*}
           while for 
           $\beta_{\ell_1}+\dots +\beta_{\ell_{2q}} = 1$
           \begin{equation*}
           \lim_{T\to +\infty}\frac{K_{\ell_1,\dots, \ell_{2q}}(T)}{T\log T}  =  2C_{\ell_1}(0)\cdots C_{\ell_{2q}}(0). 
           \end{equation*} 
           Finally for   
           $\beta_{\ell_1}+\dots +\beta_{\ell_{2q}} > 1$ 
            \begin{equation*}
           \lim_{T\to +\infty}\frac{K_{\ell_1,\dots, \ell_{2q}}(T)}{T}  = \int_{\mathbb R} C_{\ell_1}(\tau)\cdots C_{\ell_{2q}}(\tau)\,d\tau. 
           \end{equation*}      
           Thanks to technical Lemmas 4.7--4.10 in \cite{MRV21}, Lemmas 4.13-4.17 in \cite{MRV21} and the fact that 
           \begin{align*}
          &\sum_{\ell_1,\dots,\ell_{2q}\in \tilde \Lambda} \kappa_{\ell_1}\cdots \kappa_{\ell_{2q}}   C_{\ell_1}(0)\cdots C_{\ell_{2q}}(0) \int_{(\mathbb M^d)^2} \left | P^{(\alpha, \beta)}_{\ell_1}(\cos \epsilon \rho(x,y))\cdots P^{(\alpha, \beta)}_{\ell_{2q}}(\cos \epsilon \rho(x,y)) \right |\,d\nu(x)d\nu(y)   \\
          &\le \sum_{\ell_1,\dots,\ell_{2q}\in \tilde \Lambda}\kappa_{\ell_1}\cdots \kappa_{\ell_{2q}}   C_{\ell_1}(0)\cdots C_{\ell_{2q}}(0)   P^{(\alpha, \beta)}_{\ell_1}(1)\cdots P^{(\alpha, \beta)}_{\ell_{2q}}(1)\\
          &= \sum_{\ell_1,\dots,\ell_{2q}\in \tilde \Lambda}  \textnormal{dim}(\mathcal Y_{\ell_1})\cdots \textnormal{dim}(\mathcal Y_{\ell_{2q}})  C_{\ell_1}(0)\cdots C_{\ell_{2q}}(0) <+\infty
          \end{align*}
           by \eqref{var1}, 
           we can apply the Dominated Convergence Theorem as in the proofs of Proposition 4.5 and Proposition 4.12 in \cite{MRV21}, thus concluding the proof of our Proposition \ref{prop1}.
           
        \end{proof}

        \subsection{Limiting distribution}
        \begin{proof}[Proof of Proposition \ref{law2caos}]
Due to symmetry of chaotic coefficients in \eqref{alpha},  and definitions \eqref{betastar} and \eqref{setstar},
\begin{align}\label{caos2}
    \frac{\mathcal M_T(u)[2]}{T^{1-\beta_{\ell^*}}} &=  \frac{1}{T^{1-\beta_{\ell^*}}}\frac{\alpha_{2, 0,\dots, 0}(u)}{2!} \sum_{j=1}^k\int_{0}^T \int_{\mathbb M^d} H_{2}(Z_j(x,t))\,dx dt\nonumber \\
    &= \frac{1}{T^{1-\beta_{\ell^*}}}\frac{\alpha_{2, 0,\dots, 0}(u)}{2!} \sum_{j=1}^k\sum_{\ell_j\in \mathcal I^*} \sum_{m_j=1}^{\dime}\int_{0}^T  H_{2}(a_{\ell_j,m_j}^{(j)}(t))\, dt + o(1),
\end{align}
where $a_{\ell_j,m_j}^{(j)}$, $\ell_j\in \tilde \Lambda$, $m_j=1,\dots, \dime$ are $k$ i.i.d. copies of $a_{\ell,m}$, $\ell\in \tilde \Lambda$, $m=1,\dots, \dime$. 
(By $o(1)$ we mean a sequence of r.v.'s converging to zero in $L^2(\mathbb P)$.) 

From \eqref{caos2} and Theorem 1 in \cite{DM79} we have 
\begin{eqnarray}\label{caos2bis}
    \frac{\mathcal M_T(u)[2]}{T^{1-\beta_{\ell^*}}} 
    &=& \frac{1}{T^{1-\beta_{\ell^*}}}\frac{\alpha_{2, 0,\dots, 0}(u)}{2!} \sum_{j=1}^k\sum_{\ell_j\in \mathcal I^*} \sum_{m_j=1}^{\textnormal{dim}(\mathcal Y_{\ell_j})}\int_{0}^T  H_{2}(a_{\ell_j,m_j}^{(j)}(t))\, dt + o(1)\cr
    &=& \frac{\alpha_{2, 0,\dots, 0}(u)}{2!} \sum_{j=1}^k\sum_{\ell_j\in \mathcal I^*} \frac{C_{\ell_j}(0)}{a(\beta_{\ell^*})}\sum_{m_j=1}^{\textnormal{dim}(\mathcal Y_{\ell_j})}\frac{a(\beta_{\ell^*})}{T^{1-\beta_{\ell^*}}C_{\ell_j}(0)}\int_{0}^T  H_{2}(a_{\ell_j,m_j}^{(j)}(t))\, dt + o(1)\cr
    &&\mathop{\to}^d \frac{\alpha_{2, 0,\dots, 0}(u)}{2!} \sum_{j=1}^k\sum_{\ell_j\in \mathcal I^*} \frac{C_{\ell_j}(0)}{a(\beta_{\ell^*})}\sum_{m_j=1}^{\textnormal{dim}(\mathcal Y_{\ell_j})} X_{m_j;\beta_{\ell^*}},
\end{eqnarray}
where $X_{m_j;\beta_{\ell^*}}$, $j=1,\dots,k$, $m_j=1,\dots, \textnormal{dim}(\mathcal Y_{\ell_j})$, are i.i.d. copies of a Rosenblatt r.v. with parameter $\beta_{\ell^*}$ (recall Definition \ref{defRosenblatt}). 

\end{proof}

The proof of Proposition \ref{prop2} is similar to the proof of Proposition 5.3 in \cite{MRV21} hence we omit the details. 

        \appendix

        \section{Technical proofs}\label{secA1}

        \begin{proof}[Proof of Lemma \ref{lemalpha}]
For $k=1$, since $N_1\mathop{=}^d -N_1$ it is easy to check that $\forall u>0$
\begin{eqnarray*}
    \alpha_{2n_1}(u) &=& 2 \mathbb E[1_{N_1 \ge \sqrt u} H_{2n_1}(N_1)]= 2 \int_{\sqrt u}^{+\infty} H_{2n_1}(t)\phi(t)\,dt= 
    2 \phi(\sqrt u) H_{2n_1-1}(\sqrt u),
\end{eqnarray*}
where $\phi$ still denotes the standard Gaussian density (in the last equality, it is enough to use the definition of Hermite polynomials). 

The approach to follow is inspired by \cite{MPRW16}. Recall the generating function of Hermite polynomials \cite[Proposition 1.4.2]{NP12}
\begin{equation*}
    e^{xt- \frac{t^2}{2}} = \sum_{q=0}^{+\infty} H_q(x) \frac{t^q}{q!},
\end{equation*}
and, for $k\ge 2$,
consider the integral 
\begin{eqnarray}
&&\frac{1}{(2\pi)^{\frac{k}{2}}}\int_{\mathbb R^k} \mathbf{1}_{\lbrace x_1^2+\dots + x_k^2\ge u\rbrace } e^{\lambda_1 x_1-\frac{\lambda_1^2}{2}} \cdots e^{\lambda_k x_k-\frac{\lambda_k^2}{2}} e^{-\frac{x_1^2+\dots +x_k^2}{2}}\,dx_1\cdots dx_k \cr
&&= \frac{1}{(2\pi)^{\frac{k}{2}}}\int_{\mathbb R^k} \mathbf{1}_{\lbrace x_1^2+\dots + x_k^2\ge u\rbrace }  e^{-\frac{(x_1-\lambda_1)^2+\dots +(x_k-\lambda_k)^2}{2}}\, dx_1\cdots dx_k.\label{intMR}
\end{eqnarray}
The integral \eqref{intMR} coincides with $\mathbb P(X_1^2+\dots + X_k^2\ge u)$, where $X_1,\dots, X_k$ are independent Gaussian r.v.'s with variance one and mean $\mathbb E[X_i]=\lambda_i$, $i=1,\dots,k$. Note that 
\begin{equation*}
    X_1^2 +\dots + X_k^2 \sim \chi^2(k, \lambda_1^2+\dots + \lambda_k^2)
\end{equation*}
        is a non-central chi-square r.v. with $k$ degrees of freedom and noncentrality parameter $\lambda_1^2+\dots + \lambda_k^2$.

        It is known that 
        \begin{align}\label{cdfChi}
            \mathbb P(\chi^2(k,\lambda_1^2+\dots + \lambda_k^2 ) \ge u) = 
            Q_{\frac{k}{2}} \left (\sqrt{\lambda_1^2+\dots + \lambda_k^2}, \sqrt{u}\right ),
        \end{align}
        where $Q$ is the so-called $\frac{k}{2}$-th order generalized Marcum $Q$-function. 
        The canonical representation of the $\nu$-th order generalized Marcum function $Q_{\nu}$ is 
        \begin{align}\label{canonicalQ}
            Q_{\nu} (a,b) = 1- e^{-a^2/2} \sum_{l=0}^{\infty} \frac{1}{l!} \frac{\gamma(\nu + l, \frac{b^2}{2})}{\Gamma(\nu+l)} \left ( \frac{a^2}{2}\right )^l
        \end{align}
        (for $\nu, a>0$ and $b\ge 0$), where $\gamma$ is the lower incomplete Gamma function, and $\Gamma$ denotes the (standard) Gamma function. 
        Hence we can write from \eqref{cdfChi} and \eqref{canonicalQ} 
        \begin{eqnarray*}
           && \mathbb P(\chi^2(k,\lambda_1^2+\dots + \lambda_k^2 ) \ge u) = 
             1 - e^{-\frac{\lambda_1^2+\dots + \lambda_k^2}{2}} \sum_{l=0}^{+\infty} \frac{1}{l!} \frac{\gamma(\frac{k}{2} + l, \frac{u}{2})}{\Gamma(\frac{k}{2}+l)} \left ( \frac{\lambda_1^2+\dots + \lambda_k^2}{2}\right )^l\cr
            &&= 1- e^{-\frac{\lambda_1^2+\dots + \lambda_k^2}{2}} \sum_{l=0}^{+\infty} \frac{1}{l!} \frac{\gamma(\frac{k}{2} + l, \frac{u}{2})}{\Gamma(\frac{k}{2}+l)} 
            \sum_{m_1+\dots + m_k=l} \frac{1}{2^l}{l \choose m_1,\dots, m_k} \lambda_1^{2m_1}\cdots \lambda_k^{2m_k}\cr
            &&= 1- \sum_{r_1, \dots, r_k=0}^{+\infty} \frac{(-1)^{r_1+\dots + r_k}\lambda_1^{2r_1}\cdots \lambda_k^{2r_k}}{2^{r_1}\cdots 2^{r_k}r_1! \cdots r_k!} \sum_{l=0}^{+\infty} \frac{1}{l!} \frac{\gamma(\frac{k}{2} + l, \frac{u}{2})}{\Gamma(\frac{k}{2}+l)} 
            \sum_{m_1+\dots + m_k=l} \frac{1}{2^l}{l \choose m_1,\dots, m_k} \lambda_1^{2m_1}\cdots \lambda_k^{2m_k}\cr
            &&= 1- \sum_{l,r_1, \dots, r_k=0}^{+\infty} \frac{(-1)^{r_1+\dots + r_k}}{2^{r_1}\cdots 2^{r_k}r_1! \cdots r_k!}  \frac{1}{l!} \frac{\gamma(\frac{k}{2} + l, \frac{u}{2})}{\Gamma(\frac{k}{2}+l)} 
            \sum_{m_1+\dots + m_k=l} \frac{1}{2^l}{l \choose m_1,\dots, m_k} \lambda_1^{2m_1+2r_1}\cdots \lambda_k^{2m_k+2r_k}\cr
            &&= 1- \sum_{l,r_1, \dots, r_k=0}^{+\infty} \frac{(-1)^{r_1+\dots + r_k}}{2^{r_1}\cdots 2^{r_k}r_1! \cdots r_k!}  \frac{1}{l!} \frac{\gamma(\frac{k}{2} + l, \frac{u}{2})}{\Gamma(\frac{k}{2}+l)} \cr 
            &\times &
            \sum_{m_1+\dots + m_k=l} \frac{1}{2^l}{l \choose m_1,\dots, m_k}(2m_1+2r_1)!\cdots  (2m_k+2r_k)!\frac{\lambda_1^{2m_1+2r_1}}{(2m_1+2r_1)!}\cdots \frac{\lambda_k^{2m_k+2r_k}}{(2m_k+2r_k)!}.
        \end{eqnarray*}
Performing the change of variable $2t_1=2m_1+2r_1$, \dots, $2t_k=2m_k + 2r_k$, we get 
\begin{align*}
   & 1- \sum_{l,r_1, \dots, r_k=0}^{+\infty} \frac{(-1)^{r_1+\dots + r_k}}{2^{r_1}\cdots 2^{r_k}r_1! \cdots r_k!}  \frac{1}{l!} \frac{\gamma(\frac{k}{2} + l, \frac{u}{2})}{\Gamma(\frac{k}{2}+l)} \\
            &\times 
            \sum_{m_1+\dots + m_k=l} \frac{1}{2^l}{l \choose m_1,\dots, m_k}(2m_1+2r_1)!\cdots  (2m_k+2r_k)!\frac{\lambda_1^{2m_1+2r_1}}{(2m_1+2r_1)!}\cdots \frac{\lambda_k^{2m_k+2r_k}}{(2m_k+2r_k)!}\\
           & = 1- \sum_{t_1,\dots,t_k=0}^{+\infty}(2t_1)!\cdots  (2t_k)!\sum_{r_1\le t_1, \dots, r_k\le t_k}\frac{(-1)^{r_1+\dots + r_k}}{2^{r_1}\cdots 2^{r_k}r_1! \cdots r_k!}  \\
            &\times 
            \frac{1}{(t_1-r_1+\dots +t_k-r_k)!} \frac{\gamma(\frac{k}{2} + t_1-r_1+\dots +t_k-r_k, \frac{u}{2})}{\Gamma(\frac{k}{2}+t_1-r_1+\dots +t_k-r_k)} \\
            &
            \times \frac{1}{2^{t_1-r_1+\dots +t_k-r_k}}{t_1-r_1+\dots +t_k-r_k \choose t_1-r_1,\dots, t_k-r_k}
             \frac{\lambda_1^{2t_1}}{(2t_1)!}\cdots \frac{\lambda_k^{2t_k}}{(2t_k)!}.
\end{align*}
On the other hand,  
        \begin{align*}
           &\frac{1}{(2\pi)^{\frac{k}{2}}}\int_{\mathbb R^k} 1_{x_1^2+\dots + x_k^2\ge u} \sum_{t_1=0}^{+\infty} H_{t_1}(x_1) \frac{\lambda_1^{t_1}}{t_1!} \cdots \sum_{t_k=0}^{+\infty} H_{t_k}(x_k) \frac{\lambda_1^{t_k}}{t_k!}  e^{-\frac{x_1^2+\dots +x_k^2}{2}} dx_1\cdots dx_k \\
           &= \sum_{t_1,\dots,t_k=0}^{+\infty} \underbrace{\frac{1}{(2\pi)^{\frac{k}{2}}} \int_{\mathbb R^k} 1_{x_1^2+\dots + x_k^2\ge u}  H_{t_1}(x_1)  \cdots  H_{t_k}(x_k)   e^{-\frac{x_1^2+\dots +x_k^2}{2}} dx_1\cdots dx_k}_{=\alpha_{t_1,\dots, t_k}(u)} \frac{\lambda_1^{t_1}}{t_1!}\cdots \frac{\lambda_1^{t_k}}{t_k!}\\
           &= \sum_{t_1,\dots,t_k=0}^{+\infty}  \alpha_{2 t_1,\dots, 2 t_k}(u) \frac{\lambda_1^{2t_1}}{(2t_1)!}\cdots \frac{\lambda_1^{2t_k}}{(2t_k)!},
        \end{align*}
        where the last two equalities follow from \eqref{alpha} and Lemma \ref{lemodd} respectively. 
Hence necessarily 
        \begin{equation*}
            \alpha_{(0,\dots,0)}(u) = 1- \frac{\gamma(\frac{k}{2},\frac{u}{2})}{\Gamma(\frac{k}{2})},
        \end{equation*}
        and more generally
       \begin{align*}
            &\alpha_{2t_1,\dots, 2t_k}(u) =
            (2t_1)!\cdots  (2t_k)!\sum_{r_1\le t_1, \dots, r_k\le t_k}\frac{(-1)^{r_1+\dots + r_k}}{2^{r_1}\cdots 2^{r_k}r_1! \cdots r_k!}  \\
            &\times 
            \frac{1}{(t_1-r_1+\dots +t_k-r_k)!} \frac{\gamma(\frac{k}{2} + t_1-r_1+\dots +t_k-r_k, \frac{u}{2})}{\Gamma(\frac{k}{2}+t_1-r_1+\dots +t_k-r_k)}  \frac{1}{2^{t_1-r_1+\dots +t_k-r_k}}{t_1-r_1+\dots +t_k-r_k \choose t_1-r_1,\dots, t_k-r_k}.
        \end{align*}
For instance, for $t_1=q$, $t_2=\dots=t_k=0$ we have 
\begin{align*}
    \alpha_{(2q,0,\dots, 0)}(u) &= (2q)!\sum_{r_1\le q}\frac{(-1)^{r_1}}{2^{r_1}r_1! }  
            \frac{1}{2^{q-r_1}(q-r_1)!} \frac{\gamma(\frac{k}{2} + q-r_1, \frac{u}{2})}{\Gamma(\frac{k}{2}+q-r_1)}\\
            &= (2q)!\sum_{r_1\le q}\frac{(-1)^{r_1}}{q!2^{q} } 
            {q \choose r_1}\frac{\gamma(\frac{k}{2} + q-r_1, \frac{u}{2})}{\Gamma(\frac{k}{2}+q-r_1)}\\
            &= \frac{(2q)!}{q!2^{q}}\sum_{r_1\le q}(-1)^{r_1} 
            {q \choose r_1}\frac{\gamma(\frac{k}{2} + q-r_1, \frac{u}{2})}{\Gamma(\frac{k}{2}+q-r_1)}.
\end{align*}
\end{proof}

\begin{proof}[Proof of \eqref{alpha2}]
We can write
\begin{align*}
\alpha_{2,0, \dots,0}(u) &= \mathbb E \left [\mathbb E \left [1_{\sum_{j=1}^k N_j^2\ge u}  H_{2}(N_1) \Big | N_{2},\dots, N_{k} \right ]\right ] \\
&= \mathbb E \left [ {\mathbb E \left [1_{ N_1^2 + \sum_{j=2}^k x_j^2\ge u}  H_{2}(N_1) \right ]}_{ \Big |x_2=N_{2},\dots, x_k=N_{k}} \right ].
\end{align*}
Now observe that, if $u - \sum_{j=2}^k x_j^2 \le 0$, then 
$$
\mathbb E \left [1_{ N_1^2 + \sum_{j=2}^k x_j^2\ge u}  H_{2}(N_1) \right ] = \mathbb{E}[H_{2}(N_1)] = 0.
$$
On the other hand, if $u - \sum_{j=2}^k x_j^2 > 0$, 
$$
\mathbb E \left [1_{ N_1^2 + \sum_{j=2}^k x_j^2\ge u}  H_{2}(N_1) \right ]= 2\phi\left (\sqrt{u - \sum_{j=2}^k x_j^2}\right)H_1\left(\sqrt{u - \sum_{j=2}^k x_j^2}\right) > 0.
$$
Hence,
$$
\alpha_{2,0,\dots,0}(u) = \int_{\mathbb{R}^{k-1}} \mathbf{1}_{u - \sum_{j=2}^k x_j^2 > 0} \, 2\phi\left (\sqrt{u - \sum_{j=2}^k x_j^2}\right)H_1\left(\sqrt{u - \sum_{j=2}^k x_j^2}\right) \prod_{j=2}^k \phi(x_j)\, dx_2\dots dx_{k} > 0,
$$
thus concluding the proof. 

\end{proof}
\begin{remark}
    More generally, we can deduce the following formula: for $n_1\in \mathbb N$
\begin{align*}
\alpha_{2n_1,0,\dots,0}(u) 
= \mathbb{E}\left [1_{u - \sum_{j=2}^k N_j^2 > 0} \, 2\phi\left (\sqrt{u - \sum_{j=2}^k N_j^2}\right)H_{2n_1-1}\left(\sqrt{u - \sum_{j=2}^k N_j^2}\right)\right ].
\end{align*}
\end{remark}

\bibliographystyle{alpha}
\bibliography{bibfile}

\begin{thebibliography}{LRMT17b}

\bibitem[AS64]{AS64}
Milton Abramowitz and Irene~A. Stegun.
\newblock {\em Handbook of mathematical functions with formulas, graphs, and
  mathematical tables}, volume No. 55 of {\em National Bureau of Standards
  Applied Mathematics Series}.
\newblock U. S. Government Printing Office, Washington, DC, 1964.
\newblock For sale by the Superintendent of Documents.

\bibitem[CMV23]{CMV23}
Alessia Caponera, Domenico Marinucci, and Anna Vidotto.
\newblock Multiscale {CUSUM} tests for time-dependent spherical random fields.
\newblock {\em Preprint ArXiv:2305.01392}, 2023.

\bibitem[DHLM23]{DHLM23}
Krzysztof D\c{e}bicki, Enkelejd Hashorva, Peng Liu, and Zbigniew Michna.
\newblock Sojourn times of {G}aussian and related random fields.
\newblock {\em ALEA Lat. Am. J. Probab. Math. Stat.}, 20(1):249--289, 2023.

\bibitem[DM79]{DM79}
Roland~L. Dobrushin and P\'eter Major.
\newblock Non-central limit theorems for nonlinear functionals of {G}aussian
  fields.
\newblock {\em Z. Wahrsch. Verw. Gebiete}, 50(1):27--52, 1979.

\bibitem[GM75]{gin75}
Evarist Gin\'{e}~Masd\'eu.
\newblock The addition formula for the eigenfunctions of the {L}aplacian.
\newblock {\em Advances in Math.}, 18(1):102--107, 1975.

\bibitem[KT12]{KT12}
Alexander Kushpel and Sergio~A. Tozoni.
\newblock Entropy and widths of multiplier operators on two-point homogeneous
  spaces.
\newblock {\em Constr. Approx.}, 35(2):137--180, 2012.

\bibitem[Kus21]{Kus21}
Alexander Kushpel.
\newblock The {L}ebesgue constants on projective spaces.
\newblock {\em Turkish J. Math.}, 45(2):856--863, 2021.

\bibitem[LMNP24]{LMNP24}
Nikolai~N. Leonenko, Leonardo Maini, Ivan Nourdin, and Francesca Pistolato.
\newblock Limit theorems for $p$-domain functionals of {G}aussian fields.
\newblock {\em Preprint arXiv:2402.16701}, 2024.

\bibitem[LMX23]{LMX23}
Tianshi Lu, Chunsheng Ma, and Yimin Xiao.
\newblock Strong local nondeterminism and exact modulus of continuity for
  isotropic {G}aussian random fields on compact two-point homogeneous spaces.
\newblock {\em J. Theoret. Probab.}, 36(4):2403--2425, 2023.

\bibitem[LRM23]{LRM23}
Nikolai~N. Leonenko and Maria~Dolores Ruiz-Medina.
\newblock Sojourn functionals for spatiotemporal gaussian random fields with
  long memory.
\newblock {\em Journal of Applied Probability}, 60(1):148–165, 2023.

\bibitem[LRMT17a]{LRMT17}
Nikolai~N. Leonenko, Maria~Dolores Ruiz-Medina, and Murad~S. Taqqu.
\newblock Non-central limit theorems for random fields subordinated to
  gamma-correlated random fields.
\newblock {\em Bernoulli}, 23(4B):3469--3507, 2017.

\bibitem[LRMT17b]{LRT17}
Nikolai~N. Leonenko, Maria~Dolores Ruiz-Medina, and Murad~S. Taqqu.
\newblock Rosenblatt distribution subordinated to {G}aussian random fields with
  long-range dependence.
\newblock {\em Stoch. Anal. Appl.}, 35(1):144--177, 2017.

\bibitem[LTT18]{LTT18}
Nikolai~N. Leonenko, Murad~S. Taqqu, and Gyorgy~H. Terdik.
\newblock Estimation of the covariance function of {G}aussian isotropic random
  fields on spheres, related {R}osenblatt-type distributions and the cosmic
  variance problem.
\newblock {\em Electron. J. Stat.}, 12(2):3114--3146, 2018.

\bibitem[MM20]{MM20}
Chunsheng Ma and Anatoliy Malyarenko.
\newblock Time-varying isotropic vector random fields on compact two-point
  homogeneous spaces.
\newblock {\em J. Theoret. Probab.}, 33(1):319--339, 2020.

\bibitem[MPRW16]{MPRW16}
Domenico Marinucci, Giovanni Peccati, Maurizia Rossi, and Igor Wigman.
\newblock Non-universality of nodal length distribution for arithmetic random
  waves.
\newblock {\em Geom. Funct. Anal.}, 26:926–--960, 2016.

\bibitem[MRV21]{MRV21}
Domenico Marinucci, Maurizia Rossi, and Anna Vidotto.
\newblock Non-universal fluctuations of the empirical measure for isotropic
  stationary fields on {$\Bbb S^2 \times \Bbb R$}.
\newblock {\em Ann. Appl. Probab.}, 31(5):2311--2349, 2021.

\bibitem[MRV24]{MRV24+}
Domenico Marinucci, Maurizia Rossi, and Anna Vidotto.
\newblock Fluctuations of level curves for time-dependent spherical random
  fields.
\newblock {\em Ann. H. Lebesgue (to appear)}, 2024.

\bibitem[MS22]{MS22}
Vitalii Makogin and Evgeny Spodarev.
\newblock Limit theorems for excursion sets of subordinated {G}aussian random
  fields with long-range dependence.
\newblock {\em Stochastics}, 94(1):111--142, 2022.

\bibitem[NP12]{NP12}
Ivan Nourdin and Giovanni Peccati.
\newblock {\em Normal approximations with {M}alliavin calculus}, volume 192 of
  {\em Cambridge Tracts in Mathematics}.
\newblock Cambridge University Press, Cambridge, 2012.
\newblock From Stein's method to universality.

\bibitem[Sze67]{Sze67}
G\'{a}bor Szeg\H{o}.
\newblock {\em Orthogonal polynomials}, volume Vol. 23 of {\em American
  Mathematical Society Colloquium Publications}.
\newblock American Mathematical Society, Providence, RI, third edition, 1967.

\bibitem[Taq79]{Taq79}
Murad~S. Taqqu.
\newblock Convergence of integrated processes of arbitrary {H}ermite rank.
\newblock {\em Z. Wahrsch. Verw. Gebiete}, 50(1):53--83, 1979.

\bibitem[Taq75]{Taq75}
Murad~S. Taqqu.
\newblock Weak convergence to fractional {B}rownian motion and to the
  {R}osenblatt process.
\newblock {\em Z. Wahrscheinlichkeitstheorie und Verw. Gebiete}, 31:287--302,
  1974/75.

\bibitem[Wan52]{Wan52}
Hsien-Chung Wang.
\newblock Two-point homogeneous spaces.
\newblock {\em Ann. of Math. (2)}, 55:177--191, 1952.

\end{thebibliography}
\end{document}